\documentclass[12pt,a4paper]{article}
\usepackage[T1]{fontenc}
\usepackage[margin=3cm,footskip=1cm]{geometry}

\usepackage{amsmath,amssymb,amsthm,bm,mathtools,xspace,booktabs,url}
\usepackage{graphicx,etoolbox}
\usepackage{ dsfont }
\usepackage{bbold}
\usepackage{tikz}
\usepackage{ mathrsfs }
\usepackage[shortlabels]{enumitem}
\usepackage{xcolor}
\usepackage{color}

\newcommand{\dd}{\mathrm{d}}
\makeatletter
\global\let\tikz@ensure@dollar@catcode=\relax
\makeatother
\setlist{
  listparindent=\parindent,
  parsep=0pt,
}
\usepackage{amssymb}

\usepackage{hyperref}
\usepackage{cleveref}

\AtBeginEnvironment{thebibliography}{
  \small
  \setlength\itemsep{0.75em plus 0.5em minus 0.75em}%
}

\usepackage{caption}
\usepackage[raggedright]{titlesec}

\numberwithin{equation}{section}
\usepackage{amsfonts}

\theoremstyle{plain} 
\newtheorem{theorem}{Theorem}[section]

\newtheorem{Corollary}[theorem]{Corollary}

\newtheorem*{theorem*}{Theorem}

\theoremstyle{definition} 
\newtheorem{Example}[theorem]{Example}

\usepackage[above,below,section]{placeins}

\usepackage[all]{onlyamsmath}

\usepackage{multirow}

\usepackage{xcolor}
\usepackage{color}

\usepackage{mathrsfs}
\usepackage{pdfpages}

\definecolor{darkmagenta}{rgb}{0.5,0,0.5}
\definecolor{darkgreen}{rgb}{0,0.6,0}
\definecolor{darkblue}{rgb}{0,0,0.6}
\definecolor{darkred}{rgb}{0.8,0,0}
\definecolor{mellow}{rgb}{.847, 0.72, 0.525}

\begin{document}

\title{A Markov jump process associated with the matrix-exponential distribution}

\author{Oscar Peralta\footnote{The University of Adelaide, School of Mathematical Sciences, SA 5005, Australia, \texttt{oscar.peraltagutierrez@adelaide.edu.au}}
}
\date{}

\maketitle
\begin{abstract}
Let $f$ be the density function associated to a matrix-exponential distribution of parameters $(\bm{\alpha}, T,\bm{s})$. By exponentially tilting $f$, we find a probabilistic interpretation which generalises the one associated to phase-type distributions. More specifically, we show that for any sufficiently large $\lambda\ge 0$, the function $x\mapsto \left(\int_0^\infty e^{-\lambda s}f(s)\dd s\right)^{-1}e^{-\lambda x}f(x)$ can be described in terms of a Markov jump process whose generator is tied to $T$. Finally, we show how to revert the exponential tilting in order to assign a probabilistic interpretation to $f$ itself. 
\end{abstract}
\section{Introduction}
A phase-type distribution corresponds to the law of $Y:=\inf\{t\ge 0: J_t=\star\}$ where $\{J_t\}_{t\ge 0}$ is a Markov jump process with state space $\{1,\dots,p\}\cup \{\star\}$, with $\{1,\dots,p\}$ assumed to be transient states and $\{\star\}$ absorbing. If $\{J_t\}_{t\ge 0}$ has a block-partitioned initial distribution $(\bm{\pi}, 0)$ and intensity matrix given by
\begin{equation}\label{eq:MarkovPH1}\left[\begin{matrix}A&\bm{b}\\\bm{0}&0\end{matrix}\right]\quad\mbox{with}\quad\bm{b}=-A\bm{1},\end{equation}
where $\bm{0}$ represents a $p$-dimensional row vector of $0$'s and $\bm{1}$ a $p$-dimensional column vector of $1$'s, then we say that the phase-type distribution is of parameters $(\bm{\pi}, A)$. Via simple probabilistic arguments, it can be shown that the density function of a phase-type distribution of parameters $(\bm{\pi},A)$ is of the form
\begin{equation}\label{eq:PHdensity} g(x)=\bm{\pi}e^{Ax}\bm{b}, \quad x\ge 0. \end{equation}
Indeed, the vector $\bm{\pi}e^{Ax}$ yields the probabilities of $\{J_t\}_{t\ge 0}$ being in some state $\{1,\dots, p\}$ at time $x$, and $\bm{b}$ corresponds to the intensity vector of an absorption happening immediately after. Phase-type distributions were first introduced in \cite{Jensen:2011bz} with the aim of constructing a robust and tractable class of distributions on $\mathds{R}_+$ to be used in econometric problems. A more comprehensive study of phase-type distributions was carried on by Neuts \cite{NeutsPH,Neuts}, whose work popularized their use in more general stochastic models.

On the other hand, a \emph{matrix-exponential distribution} of dimension $p\ge 1$ is an absolutely continuous distribution on $(0,\infty)$ whose density function can be written as
\begin{equation}\label{eq:MEdensity} f(x)=\bm{\alpha}e^{Tx}\bm{s},\quad x\ge 0,\end{equation}
where $\bm{\alpha}=(\alpha_1,\dots,\alpha_p)$ is a $p$-dimensional row vector, $T=\{t_{ij}\}_{i,j\in\{1,\dots,p\}}$ is a $(p\times p) $-dimensional square matrix, and $\bm{s}=(s_1,\dots,s_p)^\intercal$ is a $p$-dimensional column vector, all with complex entries. If the dimension need not be specified, we refer to such a distribution simply as \emph{matrix-exponential}. It follows from (\ref{eq:PHdensity}) and (\ref{eq:MEdensity}) that the class of phase-type distributions is a subset of those that are matrix-exponential, with the inclusion being strict (see \cite{o1990characterization} for details on the latter). 

{Matrix-exponential distributions were first studied in \cite{Cox:2008hh,Cox:1955cc} through the concept of complex-valued transition probabilities. More precisely, it was shown that certain systems with complex-valued elements can be formally studied by analytical means without assigning a specific physical interpretation to their components. While their method provided mathematical rigour to systems ``driven'' by complex-valued intensity matrices, it failed to provide a physical meaning to each individual component, as opposed to the case of Markov jump processes with genuine intensity matrices.} Later on, it was proved in \cite{o1990characterization,bladt2003matrix} that matrix-exponential distributions have an interpretation in terms of a Markov process with continuous state space, as opposed to the finite-state-space one that phase-type distributions enjoy. {Even after the previous physical interpretations of matrix-exponential distributions were discovered,  properties of this class of distributions are still not as well understood as they are for its phase-type counterpart. One of the main reasons for this is that continuous state space processes are more difficult to handle, so that studying matrix-exponential distributions by physical means requires a more sophisticated framework. For example, this is the case in \cite{asmussen1999point,bean2010quasi,bean2021rapmodulated} where the theory of piecewise-deterministic Markov processes is used to study models with matrix-exponential components. Thus, having a finite-state system interpretation for matrix-exponential distributions available may potentially lead to the discovery of new properties, as it has traditionally been the case for phase-type distributions.}

{In this paper we give a physical interpretation to \emph{each element} of the parameters $(\bm{\alpha}, T, \bm{s})$ satisfying:
\begin{enumerate}
\item[\textbf{A1}.] The elements of $\bm{\alpha}$, $T$ and $\bm{s}$ are real,
\item[\textbf{A2}.] The dominant eigenvalue of $T$, denoted $\sigma_0$, is real and strictly negative.
\end{enumerate}
Since it can be shown that for a given matrix-exponential density of the form (\ref{eq:MEdensity}) the parameters $(\bm{\alpha}, T, \bm{s})$ can be chosen is such a way that \textbf{A1} and \textbf{A2} hold (see \cite{bladt1992renewal}), the interpretation that we develop essentially completes the picture laid out in \cite{Cox:2008hh,Cox:1955cc}. Our method, inspired by the recent work in \cite{vollering2020markov}, provides a transparent interpretation of $(\bm{\alpha}, T, \bm{s})$ in terms of a finite-state Markov jump process.} To do so, we employ the technique known as \emph{exponential tilting}, which means that we focus on the density proportional to $e^{-\lambda \cdot}f(\cdot)$ for large enough $\lambda>0$. {After we perform this transformation, we construct a Markov jump process on a finite state space formed by two groups: the \emph{original states} and the \emph{anti-states}, the latter being a copy of the former. Heuristically, jumps within the set of original states or within the set of anti-states occur according to the off-diagonal nonnegative ``jump intensities'' of $T$, while jumps between the original and the anti-states occur according to the negative ``jump intensities'' of $T$. Eventual absorption or termination happens, and each realization ``carries'' a positive or negative sign depending only on its initial and final state. Our main contribution is showing that this mechanism yields the exponentially-tilted matrix-exponential distribution, and by reverting the exponential tilting, providing some probabilistic insight of the original matrix-exponential distribution as well.}

The structure of the paper is as follows. In Section \ref{sec:background} we provide a brief exposition on exponential tilting and how it affects the representation of a matrix-exponential distribution. In Section \ref{sec:main} we present our main results, Theorem \ref{th:main3} and Corollary \ref{cor:alt2}, where we give a precise interpretation of an exponentially-tilted matrix-exponential density in terms of a Markov jump process. Finally, in Section \ref{sec:recovering} we provide methods to recover formulae and probabilistic interpretations for any matrix-exponential distribution based on the results on their exponentially-tilted version.
\section{Preliminaries}\label{sec:background}
Exponential tilting, also known as the \emph{Escher transform}, is a technique which transforms any probability density function $f$ with support on $[0,\infty)$ into a new probability density function $f_\lambda$ defined by
\[f_\lambda(x) = \frac{e^{-\lambda x} f(x)}{\int_{0}^\infty e^{-\lambda r} f(r)\dd r}, \quad x\ge 0,\]
where $\lambda\ge 0$ is the \emph{tilting rate}. The use of exponential tilting goes back at least to \cite{escher1932probability}, where it was used to build upon Cr\'amer's classical actuarial models \cite{Cramer:2013bs}. Later on, the exponential tilting method played a prominent role in the theory of option pricing \cite{gerber1994option}.

The exponentially-tilted version of a matrix-exponential distribution has a simple form which happens to be matrix-exponential itself. To see this, notice that if $f$ is of the form (\ref{eq:MEdensity}), then
\begin{align*}
\int_0^{\infty}e^{-\lambda r}f(r)\dd r & = \int_0^{\infty}e^{-\lambda r}(\bm{\alpha}e^{Tr}\bm{s})\dd r = \bm{\alpha}(\lambda I-T)^{-1}\bm{s},
\end{align*}
where we used the fact that $T-\lambda I$ has eigenvalues with strictly negative real parts and thus $e^{(T-\lambda I) r}$ vanishes as $r\rightarrow\infty$. Thus,
\begin{align}
f_\lambda(x)=\frac{e^{-\lambda x} (\bm{\alpha}e^{Tx}\bm{s})}{\bm{\alpha}(\lambda I-T)^{-1}\bm{s}}=\left(\frac{\bm{\alpha}}{\bm{\alpha}(\lambda I-T)^{-1}\bm{s}}\right)e^{(T-\lambda I)x}\bm{s},\quad x\ge 0,\label{eq:flambda1}
\end{align}
implying that $f_\lambda$ corresponds to the density function of a matrix-exponential distribution of parameters $\left(\tfrac{\bm{\alpha}}{\bm{\alpha}(\lambda I-T)^{-1}\bm{s}}, T-\lambda I, \bm{s}\right)$.

Recall that the parameters $(\bm{\alpha},T,\bm{s})$ need not have a probabilistic meaning in terms of a finite state space Markov chain, as opposed to the parameters associated to phase-type distributions. For instance, the parameters
\begin{equation}\label{eq:MEex1}
\bm{\alpha}=\left(1,0, 0\right), \quad T=\left[\begin{array}{rrr}
-1 & -1 & 2/3 \\
1 & -1 & -2/3 \\
0 & 0 & -1
\end{array}\right], \quad \bm{s}=\left[\begin{array}{r}
4/3 \\
2/3 \\
1
\end{array}\right]
\end{equation}
yield a valid matrix-exponential distribution whose density function is given by $f(x)=\tfrac{2}{3}e^{-x}(1+\cos (x))$, and where the dominant eigenvalue of $T$ is $-1$  (see \cite[Example 4.5.21]{bladt2017matrix} for details). In the following section we show how to assign a probabilistic meaning to the exponentially-tilted version of (\ref{eq:MEex1}), and more generally to those attaining properties \textbf{A1} and \textbf{A2}, in terms of a finite-state Markov jump process.
\section{Main results}\label{sec:main}
Let $(\bm{\alpha},T,\bm{s})$ be parameters associated to a $p$-dimensional matrix-exponential distribution which attain properties \textbf{A1} and \textbf{A2}. For $1\le i,j\le p$ denote by $t_{ij}$ the $(i,j)$-entry of $T$, and denote by $s_i$ the $i$-th entry of $\bm{s}$. For $\ell\in\{+,-\}$, define the $(p\times p)$-dimensional matrix $T^{\ell}=\{t_{ij}^{\ell}\}_{1 \le i,j\le p}$ and the $p$-dimensional column vector $\bm{s}^{\ell}=(s^\ell_1,\dots, s^\ell_p)^\intercal$ where
\begin{align*}
t_{ij}^\pm&=\max\{0,\pm t_{ij}\}\quad \forall\,i\neq j,\\
t_{ii}^\pm&=\pm \min\{0, \pm t_{ii}\}\quad\forall\, i,\mbox{ and}\\
s_i^\pm&=\max\{0,\pm s_i\}\quad\forall\, i.
\end{align*}
It follows that $T^+$ has nonnegative off-diagonal elements and nonpositive diagonal elements, $T^-$ is a nonnegative matrix, $\bm{s}^+$ and $\bm{s}^-$ are nonnegative column vectors, $T=T^+-T^-$ and $\bm{s}=\bm{s}^+-\bm{s}^-$. Now, let 
\[\lambda_0=\min\left\{r\ge 0: \sum_{j=1}^p s^+_{ij} + s^-_{ij} + s^+_i+s^-_i - \lambda \le 0\mbox{ for all }1\le i\le p\right\}.\]
For some fixed $\lambda\ge\lambda_0$, consider a (possibly) terminating Markov jump process $\{\varphi_t^\lambda\}_{t\ge 0}$  driven by the block-partitioned subintensity matrix
\begin{align}\label{eq:subint1}G=\left[\begin{array}{cccc}T^+-\lambda I & T^-&\bm{s}^+&\bm{s}^-\\
T^-& T^+-\lambda I & \bm{s}^-&\bm{s}^+\\ 
\bm{0}&\bm{0} & 0& 0\\
\bm{0}&\bm{0} & 0& 0 \end{array}\right]\end{align}
evolving on the state space $\mathcal{E}=\mathcal{E}^o\cup \mathcal{E}^a\cup\{\Delta^o\}\cup\{\Delta^a\}$ where $\mathcal{E}^o:=\{1^o,2^o,\dots,p^o\}$ and $\mathcal{E}^a=\{1^a,2^a,\dots,p^a\}$. The state space $\mathcal{E}$ may be thought as the union of two sets: a collection of \emph{original states} $\mathcal{E}^o\cup\{\Delta^o\}$ and a collection of \emph{anti-states} $\mathcal{E}^a\cup\{\Delta^a\}$, where both $\Delta^o$ and $\Delta^a$ are absorbing. In the case $\lambda>\lambda_0$, the process $\{\varphi_t^\lambda\}_{t\ge 0}$ alternates between sojourn times in $\mathcal{E}^o$ and $\mathcal{E}^a$ up until one of the following happens: (a) get absorbed into $\Delta^o$, (b) get absorbed into $\Delta^a$ or (c) undergo termination due to the defect of (\ref{eq:subint1}). If $\lambda=\lambda_0$, the states $\mathcal{E}^o\cup\mathcal{E}^a$ may or may not be transient, their status depending on the values of $T$. 

In Theorem \ref{th:main4} we establish a link between the absoprtion probabilities of $\{\varphi_t^\lambda\}_{t\ge 0}$ and the vector $e^{(T-\lambda I) x}\bm{s}$ appearing in the exponentially-tilted matrix-exponential density (\ref{eq:flambda1}). More specifically, we express each element in $e^{(T-\lambda I) x}\bm{s}$ as the sum of some positive density functions and some negative density function, where the positive density is associated to an absorption of $\{\varphi_t^\lambda\}_{t\ge 0}$ to $\Delta^o$, while the negative density function corresponds to an absorption of $\{\varphi_t^\lambda\}_{t\ge 0}$ to $\Delta^a$.
{To shorten notation, from now on we denote by $\mathds{P}_j$ ($\mathds{E}_j$) , $j\in\mathcal{E}$, the probability measure (expectation) associated to $\{\varphi^\lambda_t\}_{t\ge 0}$ conditional on the event $\{\varphi^\lambda_0=j\}$.}
\begin{theorem}\label{th:main4}
Let $\lambda\ge \lambda_0$ be such that the states $\mathcal{E}^o\cup\mathcal{E}^a$ are transient. Define
\begin{equation}\label{eq:taudef1}\tau=\inf\{x\ge 0: \varphi_x^\lambda\notin \mathcal{E}^o\cup\mathcal{E}^a\}.\end{equation}
Then, for $i\in \{1,\dots, p\}$ and $x\ge 0$,
\begin{align}\big(\bm{e}_i^\intercal e^{(T-\lambda I)x}\bm{s}\big) \dd x & = \mathds{E}_{i^o}\left[\mathds{1}\{\tau\in [x,x+\dd x]\}  \beta(\varphi_\tau^\lambda)\right]\label{eq:auxmain5}\\
&= (\bm{e}_i^\intercal,\bm{0})\exp\left(\left[\begin{matrix}T^+-\lambda I & T^-\\
T^-& T^+-\lambda I\end{matrix}\right] x\right)\left[\begin{matrix}\bm{s}\\-\bm{s}\end{matrix}\right]\dd x\label{eq:auxmain6}
,\end{align}
where $\bm{e}_i$ denotes the column vector with $1$ on its $i$-th entry and $0$ elsewhere, and $\beta(j):=\mathds{1}\{j=\Delta^o\}-\mathds{1}\{j=\Delta^a\}$. Moreover, 
\begin{align}\big(-\bm{e}_i^\intercal e^{(T-\lambda I)x}\bm{s}\big) \dd x &= \mathds{E}_{i^a}\left[\mathds{1}\{\tau\in [x,x+\dd x]\}  \beta(\varphi_\tau^\lambda)\right]\label{eq:anti5}\\
& = (\bm{0},\bm{e}_i^\intercal)\exp\left(\left[\begin{matrix}T^+-\lambda I & T^-\\
T^-& T^+-\lambda I\end{matrix}\right] x\right)\left[\begin{matrix}\bm{s}\\-\bm{s}\end{matrix}\right]\dd x\label{eq:anti6}
.\end{align}
\end{theorem}
\begin{proof}
The block structure of (\ref{eq:subint1}) implies that
\begin{align*}
\mathds{P}_{i^o}(\tau\in [x,x+\dd x], \varphi_\tau^\lambda=\Delta^o) & = (\bm{e}_i^\intercal,\bm{0})\exp\left(\left[\begin{matrix}T^+-\lambda I & T^-\\
T^-& T^+-\lambda I\end{matrix}\right] x\right)\left[\begin{matrix}\bm{s}^+\\\bm{s}^-\end{matrix}\right],\\
\mathds{P}_{i^o}(\tau\in [x,x+\dd x], \varphi_\tau^\lambda=\Delta^a) & = (\bm{e}_i^\intercal,\bm{0})\exp\left(\left[\begin{matrix}T^+-\lambda I & T^-\\
T^-& T^+-\lambda I\end{matrix}\right] x\right)\left[\begin{matrix}\bm{s}^-\\\bm{s}^+\end{matrix}\right],
\end{align*}
therefore, the r.h.s. of (\ref{eq:auxmain5}) is equal to (\ref{eq:auxmain6}). We prove that (\ref{eq:auxmain5}) holds next.

Define the collection of $(p\times p)$-dimensional matrices $\{\Phi_{oa}(x)\}_{x\ge 0}$, $\{\Phi_{ao}(x)\}_{x\ge 0}$, $\{\Phi_{oo}(x)\}_{x\ge 0}$ and $\{\Phi_{oa}(x)\}_{x\ge 0}$ by
\begin{align*}
(\Phi_{oa}(x))_{ij}& =\mathds{P}_{i^o}(\tau>x, \varphi_x^\lambda=j^a) ,\qquad(\Phi_{ao}(x))_{ij} =\mathds{P}_{i^a}(\tau>x, \varphi_x^\lambda=j^o) ,\\
(\Phi_{oo}(x))_{ij}& =\mathds{P}_{i^o}(\tau>x, \varphi_x^\lambda=j^o) ,\qquad(\Phi_{aa}(x))_{ij} =\mathds{P}_{i^a}(\tau>x, \varphi_x^\lambda=j^a) ,
\end{align*}
for all $i,j\in\{1,\dots,p\}$. By the symmetry of the subintensity matrix $G$ it is clear that for all $x\ge 0$, $\Phi_{oa}(x)=\Phi_{ao}(x)$ and $\Phi_{oo}(x)=\Phi_{aa}(x)$, even if their probabilistic interpretations differ. For all $x\ge 0$, let $\Phi_{o}(x):=\Phi_{oo}(x)-\Phi_{oa}(x)$ and $\Phi_{a}(x):=\Phi_{aa}(x)-\Phi_{ao}(x)$. Define
\[\gamma=\inf\{r \ge 0: \varphi_r^\lambda\notin\mathcal{E}^o\}.\]
Then, for $i,j\in\{1,\dots,p\}$
\begin{align}
\bm{e}_i^\intercal\Phi_{o}(x)\bm{e}_j&=\mathds{P}_{i^o}(\tau>x, \varphi_x^\lambda=j^o) - \mathds{P}_{i^o}(\tau>x, \varphi_x^\lambda=j^a)\nonumber\\
& = \left\{\mathds{P}_{i^o}(\gamma>x, \varphi_x^\lambda=j^o) + \int_{0}^x \mathds{P}_{i^o}(\gamma\in[r,r+\dd r], \tau>x, \varphi_x^\lambda=j^o)\right\}\nonumber\\
 &\quad -\left\{\mathds{P}_{i^o}(\gamma>x, \varphi_x^\lambda=j^a) + \int_{0}^x \mathds{P}_{i^o}(\gamma\in[r,r+\dd r], \tau>x, \varphi_x^\lambda=j^a)\right\}\nonumber\\
& = \mathds{P}_{i^o}(\gamma>x, \varphi_x^\lambda=j^o)\nonumber\\
&\quad + \int_{0}^x \sum_{k=1}^p \mathds{P}_{i^o}(\gamma\in[r,r+\dd r], \varphi_\gamma^\lambda=k^a) \mathds{P}_{k^a}(\tau>x-r, \varphi_{x-r}^\lambda=j^o)\nonumber \\
&\quad - \int_{0}^x \sum_{k=1}^p \mathds{P}_{i^o}(\gamma\in[r,r+\dd r], \varphi_\gamma^\lambda=k^a) \mathds{P}_{k^a }(\tau>x-r, \varphi_{x-r}^\lambda=j^a),\label{eq:aux7}
\end{align}
where in the last equality we used that $\{\gamma>x, \varphi_x^\lambda=j^a\}=\varnothing$ and the Markov property of $\{\varphi_{x}^\lambda\}_{x\ge 0}$. Note that all the elements in (\ref{eq:aux7}) correspond to transition probabilities or intensities that can be expressed in matricial form as follows:
\begin{align*}
\mathds{P}_{i^o}(\gamma>x, \varphi_x^\lambda=j^o)&=\bm{e}_i^\intercal e^{(T^+-\lambda I)x}\bm{e}_j,\\
\mathds{P}_{i^o}(\gamma\in[r,r+\dd r], \varphi_\gamma^\lambda=k^a)&= \bm{e}_i^\intercal e^{(T^+-\lambda I)r}T^-\bm{e}_k\dd r,\\
\mathds{P}_{k^a}(\tau>x-r, \varphi_{x-r}^\lambda=j^o) & = \bm{e}_k^\intercal \Phi_{ao}(x-r)\bm{e}_j,\\
\mathds{P}_{k^a}(\tau>x-r, \varphi_{x-r}^\lambda=j^a) & = \bm{e}_k^\intercal \Phi_{aa}(x-r)\bm{e}_j.
\end{align*}
Substituting these expressions into (\ref{eq:aux7}) gives
\begin{align*}
\bm{e}_i^\intercal\Phi_{o}(x)\bm{e}_j & = \bm{e}_i^\intercal e^{(T^+-\lambda I)x}\bm{e}_j + \int_{0}^x \sum_{k=1}^p \left(\bm{e}_i^\intercal e^{(T^+-\lambda I)r}T^-\bm{e}_k\right) \left(\bm{e}_k^\intercal \Phi_{ao}(x-r)\bm{e}_j\right) \dd r\\
& \quad- \int_{0}^x \sum_{k=1}^p \left(\bm{e}_i^\intercal e^{(T^+-\lambda I)r}T^-\bm{e}_k\right) \left(\bm{e}_k^\intercal \Phi_{aa}(x-r)\bm{e}_j\right)\dd r\\
& = \bm{e}_i^\intercal\left(e^{(T^+-\lambda I)x} + \int_{0}^xe^{(T^+-\lambda I)r}T^-\left[\Phi_{ao}(x-r)-\Phi_{aa}(x-r)\right]\dd r\right)\bm{e}_j\\
& = \bm{e}_i^\intercal\left(e^{(T^+-\lambda I)x} + \int_{0}^xe^{(T^+-\lambda I)r}T^-\left[\Phi_{oa}(x-r)-\Phi_{oo}(x-r)\right]\dd r\right)\bm{e}_j\\
& = \bm{e}_i^\intercal\left(e^{(T^+-\lambda I)x} + \int_{0}^xe^{(T^+-\lambda I)r}(-T^-)\Phi_{o}(x-r)\dd r\right)\bm{e}_j,
\end{align*}
so that $\{\Phi_{o}(x)\}_{x\ge 0}$ is the bounded solution to the matrix-integral equation
\[\Phi_{o}(x)=e^{(T^+-\lambda I)x} + \int_{0}^xe^{(T^+-\lambda I)r}(-T^-)\Phi_{o}(x-r)\dd r.\]
By \cite[Theorem 3.10]{bean2021rapmodulated},  
\[\Phi_{o}(x)=e^{[(T^+-\lambda I)+(-T^-)]x}=e^{(T-\lambda I)x}.\]
The Markov property implies that
\begin{align*}
&\mathds{P}_{i^o}(\tau\in [x,x+\dd x], \varphi_\tau^\lambda=\Delta^o)\\
&\quad = \sum_{k=1}^p \mathds{P}_{i^o}(\tau>x, \varphi_x^\lambda=k^o)\mathds{P}_{k^o}(\tau\in[x,x+\dd x], \varphi_\tau^\lambda=\Delta^o)\\
&\quad\quad+\sum_{k=1}^p\mathds{P}_{i^o}(\tau>x, \varphi_x^\lambda=k^a)\mathds{P}_{k^a}(\tau\in[x, x+\dd x], \varphi_\tau^\lambda=\Delta^o)\\
&\quad = \sum_{k=1}^p (\bm{e}_i^\intercal\Phi_{oo}(x)\bm{e}_k)(\bm{e}_k^\intercal \bm{s}^+)\dd x + \sum_{k=1}^p (\bm{e}_i^\intercal\Phi_{oa}(x)\bm{e}_k)(\bm{e}_k^\intercal \bm{s}^-\dd x)\\
&\quad = \bm{e}_i^\intercal (\Phi_{oo}(x)\bm{s}^+ + \Phi_{oa}(r)\bm{s}^-)\dd x.
\end{align*}
Similarly,
\[\mathds{P}_{i^o}(\tau\in [x,x+\dd x], \varphi_\tau^\lambda=\Delta^a) = \bm{e}_i^\intercal (\Phi_{oa}(x)\bm{s}^+ + \Phi_{oo}(x)\bm{s}^-)\dd x.\]
Thus,
\begin{align*}
&\mathds{P}_{i^o}(\tau\in [x,x+\dd x], \varphi_\tau^\lambda=\Delta^o) - \mathds{P}_{i^o}(\tau\in [x,x+\dd x], \varphi_\tau^\lambda=\Delta^a)\\
&\quad = \bm{e}_i^\intercal ([\Phi_{oo}(x)\bm{s}^+ + \Phi_{oa}(x)\bm{s}^-]-[\Phi_{oa}(x)\bm{s}^+ + \Phi_{oo}(x)\bm{s}^-])\dd x\\
&\quad = \bm{e}_i^\intercal ([\Phi_{oo}(x)-\Phi_{oa}(x)]\bm{s}^+ - [\Phi_{oo}(x)-\Phi_{oa}(x)]\bm{s}^-)\dd x\\
& \quad = \bm{e}_i^\intercal\Phi_{o}(x)\bm{s}\dd x = \bm{e}_i^\intercal e^{(T-\lambda I)x}\bm{s}\dd x,
\end{align*}
so that (\ref{eq:auxmain5}) holds. Analogous arguments follow for (\ref{eq:anti5}) and (\ref{eq:anti6}), which completes the proof.
\end{proof}

Heuristically, Equations (\ref{eq:auxmain5}) and (\ref{eq:anti5}) imply that initiating $\{\varphi^\lambda_t\}_{t\ge 0}$ in the anti-state $i^a$ has the opposite effect, in terms of sign, to initiating in the original state $i^o$. In the following we exploit this fact to provide a probabilistic interpretation not only for the elements of $e^{(T-\lambda I)x}\bm{s}$, but for the exponentially-tilted matrix-exponential density $\bm{\alpha} e^{(T-\lambda I)x}\bm{s}$.

Define $w^+$ and $w^-$ by
\[w^\pm=\sum_{i=1}^p \max\{0,\pm\alpha_i\},\]
and define $\bm{\alpha}^+=(\alpha_1^+,\dots,\alpha_p^+)$ and $\bm{\alpha}^-=(\alpha_1^-,\dots,\alpha_p^-)$ by
\begin{align*}
\alpha^{\pm}_{i}=\left\{\begin{array}{ccc}\frac{1}{w^\pm}\max\{0,\pm\alpha_i\}&\mbox{if}&w^\pm>0\\
0&\mbox{if}& w^\pm=0.\end{array}\right.
\end{align*}
If $w^\pm>0$ then $\bm{\alpha}^\pm$ is a probability vector and in general,
\begin{equation}\label{eq:alphadec1}\bm{\alpha}=w^+\bm{\alpha}^+ - w^-\bm{\alpha}^-.\end{equation}
In some sense, $(w^+ + w^-)^{-1}\bm{\alpha}$ can be thought as a mixture of the probability vectors $\bm{\alpha}^+$ and $\bm{\alpha}^-$, with the latter contributing ``negative mass''. Fortunately, this ``negative mass'' in the context of $\bm{\alpha} e^{(T-\lambda I)x}\bm{s}$ can be given a precise probabilistic interpretation by means of anti-states as follows. 
\begin{theorem}\label{th:main3}
Let $f_\lambda(x)=(\bm{\alpha}(\lambda I-T)^{-1}\bm{s})^{-1}\bm{\alpha}e^{(T-\lambda I)x}\bm{s}$, $x\ge 0$, be the density of the exponentially-tilted matrix-exponential distribution of parameters $(\bm{\alpha},T,\bm{s})$. Define the vectors 
\[\widehat{\bm{\alpha}}^+:= \tfrac{w^+}{w^++w^-}\bm{\alpha}^+\quad\mbox{and}\quad\widehat{\bm{\alpha}}^-:= \tfrac{w^-}{w^++w^-}\bm{\alpha}^-,\] and suppose $\varphi_0^\lambda\sim (\widehat{\bm{\alpha}}^+, \widehat{\bm{\alpha}}^-)$. Then, 
\begin{align}
f_\lambda(x) \dd x &= \frac{w^++w^-}{\bm{\alpha}(\lambda I-T)^{-1}\bm{s}}\mathds{E}\left[\mathds{1}\{\tau\in [x,x+\dd x]\}  \beta(\varphi_\tau^\lambda)\right]\label{eq:gen1}\\
& =  \frac{w^++w^-}{\bm{\alpha}(\lambda I-T)^{-1}\bm{s}}\left((\widehat{\bm{\alpha}}^+,\widehat{\bm{\alpha}}^-)\exp\left(\left[\begin{matrix}T^+-\lambda I & T^-\\
T^-& T^+-\lambda I\end{matrix}\right] x\right)\left[\begin{matrix}\bm{s}\\-\bm{s}\end{matrix}\right]\right)\dd x,\label{eq:gen2}
\end{align}
where $\tau$ and $\beta(\cdot)$ are defined as in Theorem \ref{th:main4}.
\end{theorem}
\begin{proof}
Equation (\ref{eq:alphadec1}) implies that
\begin{align}
f_\lambda(x)&= \frac{1}{\bm{\alpha}(\lambda I-T)^{-1}\bm{s}}\left(\sum_{i=1}^p w^+\alpha^+_i\left(\bm{e}_i^\intercal e^{(T-\lambda I)x}\bm{s}\right) + \sum_{i=1}^p w^-\alpha^-_i\left(-\bm{e}_i^\intercal e^{(T-\lambda I)x}\bm{s}\right)\right)\nonumber\\
& = \frac{w^++w^-}{\bm{\alpha}(\lambda I-T)^{-1}\bm{s}}\left(\sum_{i=1}^p \tfrac{w^+}{w^++w^-}\alpha^+_i \left(\bm{e}_i^\intercal e^{(T-\lambda I)x}\bm{s}\right) + \sum_{i=1}^p \tfrac{w^-}{w^++w^-}\alpha^-_i\left(-\bm{e}_i^\intercal e^{(T-\lambda I)x}\bm{s}\right)\right)\label{eq:auxalphabars1}.
\end{align}
Equality (\ref{eq:gen1}) follows from (\ref{eq:auxalphabars1}), (\ref{eq:auxmain5}) and (\ref{eq:anti5}). Equality (\ref{eq:gen2}) follows from (\ref{eq:auxalphabars1}), (\ref{eq:auxmain6}) and (\ref{eq:anti6}).
\end{proof}
\begin{Example}\label{ex:tilted1}
Let $(\bm{\alpha},T,\bm{s})$ be the matrix-exponential parameters corresponding to (\ref{eq:MEex1}). {As noted previously, these parameters by themselves lack a probabilistic interpretation, so we apply Theorem \ref{th:main4} to construct one.} For such parameters we take the tilting parameter $\lambda:=\lambda_0=2$, leading to the block-partitioned matrices
\begin{align}
\left[\begin{matrix}T^+-\lambda I & T^-\\
T^-& T^+-\lambda I\end{matrix}\right] =\left[\begin{array}{rrr|rrr}
-3&0&2/3&0&1&0\\
1&-3&0&0&0&2/3\\
0&0&-3&0&0&0\\
\hline
0&1&0&-3&0&2/3\\
0&0&2/3&1&-3&0\\
0&0&0&0&0&-3
\end{array} \right],\nonumber
\end{align}
\[(\widehat{\bm{\alpha}}^+,\widehat{\bm{\alpha}}^-)=(1,0,0,0,0,0),\qquad\left[\begin{matrix}\bm{s}\\-\bm{s}\end{matrix}\right]=\left[\begin{array}{r}4/3\\2/3\\1\\\hline-4/3\\-2/3\\-1\end{array}\right],\]
and $w^+=1$, $w^-=0$ and $\bm{\alpha}(\lambda I-T)^{-1}\bm{s}=4$. We can then verify that
\begin{align*}
&\frac{w^++w^-}{\bm{\alpha}(\lambda I-T)^{-1}\bm{s}}\left((\widehat{\bm{\alpha}}^+,\widehat{\bm{\alpha}}^-)\exp\left(\left[\begin{matrix}T^+-\lambda I & T^-\\
T^-& T^+-\lambda I\end{matrix}\right] x\right)\left[\begin{matrix}\bm{s}\\-\bm{s}\end{matrix}\right]\right)\\
&\qquad=\frac{1}{6} e^{-3x}(1+\cos(x))=\frac{e^{-2x}}{4}\left(\frac{2}{3} e^{-x}(1+\cos(x))\right),
\end{align*}
the latter corresponding to the exponentially-tilted matrix-exponential density function $f(x)=\tfrac{2}{3} e^{-x}(1+\cos(x))$.\qed
\end{Example}
A probabilistic interpretation of $f_\lambda$ alternative to that of (\ref{eq:gen1}) is the following.
\begin{Corollary}\label{cor:alt2}
Define $\bm{d}= (d_1,\dots,d_p)^\intercal:=-(T^+-\lambda I)\bm{1} - T^-\bm{1}$ to be the termination intensities vector from $\mathcal{E}^o$ or $\mathcal{E}^a$, and define $\bm{q}^\pm=(q_1^\pm,\dots,q_p^\pm)^\intercal$ by
\[q^\pm_i=\left\{\begin{array}{ccc}\frac{s^\pm_i}{d_i}&\mbox{if} &d_i>0,\\ 0& \mbox{if}&d_i=0.\end{array}\right.\]
Let $\bar{q}:\mathcal{E}^o\cup\mathcal{E}^a\mapsto \mathds{R}$ be defined by
\begin{align*}
\bar{q}(i^o)=q^+_i - q^-_i\quad\mbox{ and }\quad\bar{q}(i^a)=q^-_i - q^+_i\quad\mbox{ for }\quad i\in\{1,\dots, p\}.
\end{align*}
Then,
\begin{align}
f_\lambda(x) \dd x &= \left(\frac{w^++w^-}{\bm{\alpha}(\lambda I-T)^{-1}\bm{s}}\right)\mathds{E}\left[\mathds{1}\{\tau\in [x,x+\dd x]\} \bar{q}\big(\varphi^\lambda_{\tau^-}\big)\right]\label{eq:cor1}
\end{align}
where $\{\varphi^\lambda_t\}_{t\ge 0}$ and $\tau$ are defined as in Theorem \ref{th:main3}.\end{Corollary}
\begin{proof}
First, notice that the jump mechanism of $\{\varphi^\lambda_t\}_{t\ge 0}$ described in (\ref{eq:subint1}) imply that for $i\in\{1,\dots,p\}$,
\begin{align*}
\mathds{P}(\varphi^\lambda_\tau=\Delta^o\mid \tau, \varphi^\lambda_{\tau^-}=i^o)&=q_i^+,\\
\mathds{P}(\varphi^\lambda_\tau=\Delta^a\mid \tau, \varphi^\lambda_{\tau^-}=i^o)&=q_i^-,\\
\mathds{P}(\varphi^\lambda_\tau=\Delta^o\mid \tau, \varphi^\lambda_{\tau^-}=i^a)&=q_i^-,\\
\mathds{P}(\varphi^\lambda_\tau=\Delta^a\mid \tau, \varphi^\lambda_{\tau^-}=i^a)&=q_i^+,
\end{align*}
which in turn implies that
\begin{align*}
\mathds{E}\left[\beta(\varphi_\tau^\lambda)\mid \tau,\varphi^\lambda_{\tau^-}\right]=\mathds{P}\left(\varphi_\tau^\lambda=\Delta^o\mid \tau,\varphi^\lambda_{\tau^-}\right) - \mathds{P}\left(\varphi_\tau^\lambda=\Delta^a\mid \tau,\varphi^\lambda_{\tau^-}\right)=\bar{q}(\varphi^\lambda_{\tau^-}).
\end{align*}
Consequently,
\begin{align*}
\mathds{E}\left[\mathds{1}\{\tau\in [x,x+\dd x]\}  \beta(\varphi_\tau^\lambda)\right]&=\mathds{E}\left[\mathds{E}\left[\mathds{1}\{\tau\in [x,x+\dd x]\}  \beta(\varphi_\tau^\lambda)\mid \tau,\varphi^\lambda_{\tau^-}\right]\right]\\
& = \mathds{E}\left[\mathds{1}\{\tau\in [x,x+\dd x]\}\mathds{E}\left[\beta(\varphi_\tau^\lambda)\mid \tau,\varphi^\lambda_{\tau^-}\right]\right]\\
& =\mathds{E}\left[\mathds{1}\{\tau\in [x,x+\dd x]\} \bar{q}\big(\varphi^\lambda_{\tau^-}\big)\right]
\end{align*}
and the result follows from (\ref{eq:gen1}).
\end{proof}
{Though closely related, the interpretation provided by Corollary \ref{cor:alt2} is more suitable than that of Theorem \ref{th:main3} for Monte Carlo applications. Indeed, a realization of $\{\varphi^\lambda_t\}_{t\ge 0}$ may get absorbed in $\Delta^o$, $\Delta^a$ or terminated. If termination is the case, such realization contributes nothing to the term in the r.h.s. of (\ref{eq:gen1}). In contrast, by observing the process until its exit time of $\mathcal{E}^o\cup\mathcal{E}^a$ and disregard its landing point as in Corollary \ref{cor:alt2}, we make sure that each realization contributes towards the mass in the r.h.s. of (\ref{eq:cor1}).}
\section{Recovering the untilted distribution}\label{sec:recovering}
Once the exponentially-tilted density $f_\lambda$ of a matrix-exponential distribution of parameters $(\bm{\alpha},T,\bm{s})$ has a tractable known form, say as in (\ref{eq:gen2}), in principle it is straightforward to recover the original untilted density $f$ by taking
\begin{align}f(x)& =(\bm{\alpha}(\lambda I-T)\bm{s})e^{\lambda x} f_{\lambda}(x)\nonumber\\
& = (w^++w^-)(\widehat{\bm{\alpha}}^+,\widehat{\bm{\alpha}}^-)\exp\left(\left[\begin{matrix}T^+ & T^-\\
T^-& T^+\end{matrix}\right] x\right)\left[\begin{matrix}\bm{s}\\-\bm{s}\end{matrix}\right],\qquad x\ge 0.\label{eq:untilted1}\end{align}
While (\ref{eq:untilted1}) is a legitimate matrix-exponential representation of $f$, it has two drawbacks:
\begin{enumerate}
	\item The matrix $\left[\begin{smallmatrix}T^+ & T^-\\
T^-& T^+\end{smallmatrix}\right]$ may no longer be a subintensity matrix.
\item The dominant eigenvalue of $\left[\begin{smallmatrix}T^+ & T^-\\
T^-& T^+\end{smallmatrix}\right]$ may be nonnegative.
\end{enumerate}
The first item may impact the probabilistic interpretation of $f$, while the second one may make integration of certain functions (with respect to the density $f$) more difficult to handle. For instance, in the context of Example \ref{ex:tilted1}, the matrix
\[\left[\begin{matrix}T^+ & T^-\\
T^-& T^+ \end{matrix}\right] =\left[\begin{array}{rrr|rrr}
-1&0&2/3&0&1&0\\
1&-1&0&0&0&2/3\\
0&0&-1&0&0&0\\
\hline
0&1&0&-1&0&2/3\\
0&0&2/3&1&-1&0\\
0&0&0&0&0&-1
\end{array} \right]\]
is not a subintensity matrix since some row sums are strictly positive, and it has $0$ as its dominant eigenvalue. Having $0$ as an eigenvalue implies that some entries of $\exp\left(\left(\begin{smallmatrix}T^+ & T^-\\
T^-& T^+\end{smallmatrix}\right)x\right)$ can potentially be of order $e^{0\cdot x}=1$, meaning that the matrix-integral
\begin{equation}\label{eq:intexpanded1}\int_{0}^\infty h(x) \exp\left(\left[\begin{smallmatrix}T^+ & T^-\\
T^-& T^+\end{smallmatrix}\right]x\right)\dd x\end{equation}
may only be well-defined for functions $h:\mathds{R}_+\mapsto\mathds{R}_+$ that decrease to $0$ fast enough. In comparison, $\exp(Tx)$ with $T$ as in (\ref{eq:MEex1}) has entries of at most $e^{\sigma_0 x}=e^{-x}$, so that 
\begin{equation}\label{eq:intoriginal1}\int_{0}^\infty h(x) \exp\left(Tx\right)\dd x\end{equation}
is well-defined and finite for every function $h:\mathds{R}_+\mapsto\mathds{R}_+$ of the order $e^{\sigma x}$ for any $\sigma<1$. This apparent disagreement between the applicability of (\ref{eq:intexpanded1}) and (\ref{eq:intoriginal1}) vanishes when we multiply $\exp\left(\left[\begin{smallmatrix}T^+ & T^-\\
T^-& T^+\end{smallmatrix}\right]x\right)$ with $\left[\begin{smallmatrix}\bm{s}\\ -\bm{s}\end{smallmatrix}\right]$. Indeed, in the context of Example \ref{ex:tilted1} it can be verified that the elements of the vector $\exp\left(\left[\begin{smallmatrix}T^+ & T^-\\
T^-& T^+\end{smallmatrix}\right]x\right)\left[\begin{smallmatrix}\bm{s}\\ -\bm{s}\end{smallmatrix}\right]$ are at most of order $e^{-x}$, with the higher order terms of $\exp\left(\left[\begin{smallmatrix}T^+ & T^-\\
T^-& T^+\end{smallmatrix}\right]x\right)$ cancelling each other when we multiply the matrix-function by $\left[\begin{smallmatrix}\bm{s}\\ -\bm{s}\end{smallmatrix}\right]$. In the general case, this ``cancellation'' of higher order terms than those of $e^{\sigma_0 x}$ occuring \emph{after} the matrix multiplication can be directly deduced from Theorem \ref{th:main4}.

In terms of expectations, (\ref{eq:gen1}) and (\ref{eq:cor1}) provide alternative ways to recover properties of any matrix-exponential density $f$ of parameters $(\bm{\alpha},T,\bm{s})$ in terms of the exponentially-tilted density $f_\lambda$. Indeed, for any function $h:\mathds{R}_+\mapsto\mathds{R}_+$ with $\int_0^\infty h(x)f(x)\dd x<\infty$, we have that
\begin{align}
\int_0^\infty h(x)f(x)\dd x & = (\bm{\alpha}(\lambda I-T)\bm{s})\int_0^\infty h(x)e^{\lambda x}f(x)\dd x \\
& = (w^++w^-)\mathds{E}\left[ h(\tau)e^{\lambda \tau} \beta(\varphi^\lambda_\tau)\right]\label{eq:probunt1}\\
& = (w^++w^-)\mathds{E}\left[ h(\tau)e^{\lambda \tau} \bar{q}(\varphi^\lambda_{\tau^-})\right],\label{eq:probunt2}
\end{align}
where $\{\varphi^\lambda_t\}_{t\ge 0}$ and $\tau$ are as in Theorem \ref{th:main3}. Notice that, as oppossed to the formula in (\ref{eq:untilted1}), representations (\ref{eq:probunt1}) and (\ref{eq:probunt2}) still have probabilistic interpretations in terms of the Markov jump process $\{\varphi^\lambda_t\}_{t\ge 0}$.

\section*{Acknowledgement} 
The author acknowledges the funding of the Australian Research Council Discovery Project DP180103106.

\bibliographystyle{abbrv}
\bibliography{oscar}
\end{document}